\documentclass[12pt]{article}
\usepackage{amssymb,amsbsy,amsmath,amsfonts,amssymb,amscd,amsthm,verbatim}
\usepackage[T1]{fontenc}
\usepackage[utf8]{inputenc}
\usepackage{graphicx}
\usepackage[all]{xy}
\usepackage{pdfpages}
\usepackage{verbatim}
\usepackage[linktocpage]{hyperref}

\usepackage{selinput} 

\usepackage{latexsym, amssymb, amsthm}
\usepackage{dsfont}
\usepackage{color}
\usepackage{enumerate}

\newcommand{\RR}{\mathds{R}}
\newcommand{\CC}{\mathds{C}}

\newcommand{\norm}[1]{\|#1\|}
\newcommand{\abs}[1]{\left| #1 \right|}
\newcommand{\bee}{\begin{equation}}
\newcommand{\eee}{\end{equation}}

\newcommand{\q}{\quad}
\newcommand{\qq}{\quad\quad}

\newcommand{\nf}{\infty}

\newcommand{\al}{\alpha}

\newcommand{\ga}{\gamma}

\newcommand{\de}{\delta}
\newcommand{\De}{\Delta}

\newcommand{\ve}{\varepsilon}

\newcommand{\la}{\lambda}

\newcommand{\si}{\sigma}

\newcommand{\vp}{\varphi}

\newcommand{\Om}{\Omega}

\newcommand{\li}{L^{\infty}}

\newcommand{\loc}{\textup{loc}}

\newcommand{\real}{\textup{Re}\, }
\newcommand{\imag}{\textup{Im}\, }

\newcommand{\f}{\frac}

\newcommand{\lab}{\label}

\DeclareMathAlphabet\gothic{U}{euf}{m}{n}

\newtheorem{theorem}{Theorem}
\newtheorem{lemma}[theorem]{Lemma}
\newtheorem{proposition}[theorem]{Proposition}

\newtheorem{definition}[theorem]{Definition}

\newtheorem{remark}[theorem]{Remark}

\numberwithin{theorem}{section}
\numberwithin{equation}{section}

\date{}
\begin{document} 
\title{Interpolation for analytic families of multilinear operators on metric measure spaces}
\author{ Loukas Grafakos and El Maati Ouhabaz }
\maketitle     

\begin{abstract}
Let $(X_j, d_j, \mu_j)$, $j=0,1,\dots , m$ be  metric measure spaces.  
Given $0 < p^\kappa \le \infty$ for $\kappa = 1, \dots, m$ and  an analytic family of multilinear operators 
\[
T_z: L^{p^1}(X_1)\times\cdots \times L^{p^m}(X_m) \to L^1_{\loc}(X_0), 
\]
for $z$ in the complex unit strip, we prove a  theorem in the spirit of  Stein's complex interpolation for   analytic  families. Analyticity and our admissibility condition are defined in the weak 
(integral) sense  and relax   the  pointwise  definitions  given  in \cite{GrafakosMastylo}. 
Continuous functions with compact support are natural dense subspaces of   
Lebesgue spaces over  metric measure spaces and we assume the operators $T_z$ 
 are initially defined on them.  
Our main lemma concerns the approximation of  continuous functions with compact support    
by similar functions  that depend analytically in an auxiliary parameter $z$.  
An application of the main theorem 
concerning  bilinear estimates for Schr\"odinger operators on $L^p$   
is  included.
\end{abstract}

\vspace{5mm}


\vspace{15mm}

\noindent
{\bf Home institutions:}    \\[3mm]
\begin{tabular}{@{}cl@{\hspace{10mm}}cl}
 & Department of Mathematics  & 
   & Institut de Math\'ematiques de Bordeaux \\
& University of Missouri   & 
  & Universit\'e de Bordeaux, UMR 5251,  \\
& Columbia, MO 65203 & 
  &351, Cours de la Lib\'eration. 33405 Talence  \\
& USA & 
  &  France \\
& grafakosl@missouri.edu  & 
 & Elmaati.Ouhabaz@math.u-bordeaux.fr \\[8mm]
\end{tabular}



\section{Introduction}\label{sec1}

Interpolation between function spaces plays a fundamental role in many areas of analysis 
such as harmonic, complex, and functional analysis, as well as in  PDE.  The most common interpolation 
theorems are the ones of Riesz-Thorin, Marcinkiewicz, and Stein.  Unlike the first two results which    concern a single linear operator, Stein's  interpolation theorem for   analytic  families of linear operators is formulated for families which vary analytically in an auxiliary parameter. In this way it covers and supersedes the case of a single operator, it is more flexible, and finds a variety of applications.  

In recent years, there is an increasing interest for multilinear analysis. In this setting, it is of interest to have interpolation theorems analogous to those for linear operators. The primary purpose of this article is to prove  a version of Stein's interpolation
 theorem for multilinear operators. 
Our interpolation result   is given in the context of Lebesgue spaces over
 metric measure spaces $(X_j, d_j, \mu_j) $, $j=0,1,\dots , m$ in which   balls have finite measure.  
 Such spaces have  nice subspace of dense functions such as the spaces of 
 continuous functions with compact support $\mathcal C_c(X_j)$. 
 We consider a family of multilinear operators $T_z$ for $z$ in the unit strip ${\mathbf S}=\{z\in \CC:\,\, 0 \le \textup{Re } z \le 1\}$. This family is analytic in 
 the following  sense: For every
 $f^j $ in $\mathcal C_c(X_j)$, $j=1,\dots , m$ and $w$ bounded  
 function 
with compact support on $X_0$   the mapping   
\begin{equation}\label{eq00}
z\mapsto  \int_{X_0}  T_z(f^1,\dots, f^m)\,  w \,d\mu_0
\end{equation}
is analytic in
${\mathbf S}$ and
continuous on its closure. The operators $T_z$ are taking values in the space of 
locally integrable functions on $X_0$ and satisfy the admissibility condition:  there exists a constant $\ga$ with
$0\le \ga<\pi$ and an $s \in [1, \infty]$  
 such that for any  $f^j$ in $\mathcal C_c(X_j)$ and  
 every compact subset $K$ of $X_0$ 
 there is   a constant $ C(f^1,\dots, f^m,K )$ such that
\begin{equation}\label{eq01}
\log \bigg[ \int_K | T_z(f^1,\dots, f^m)|^sd\mu_0 \bigg]^{1/s}  \le C(f^1,\dots, f^m,K )\, e^{ \ga |\textrm{Im}\,z |}, 
\q  z\in \overline{\mathbf S}.
\end{equation}
The initial estimates are of  the form 
\[
\big\|T_{j+ iy}(f^1,\dots, f^m)\big\|_{L^{q_j}(X_0) }   \le  \,\, B_j M_j (y )
\prod_{\kappa=1}^m \|f^\kappa\|_{L^{p_j^\kappa}(X_j)},\quad j\in \{0,1\},\,\,\, y\in \mathbb R,
\]
where $B_j>0$ and  $ 0<p_j^\kappa\le \infty$.  Then we prove that for $\theta \in (0,1)$ and 
\begin{equation*} \frac{1}{ p^\kappa}=\frac{1-\theta}{ p_0^\kappa}+
\frac{\theta  }{p_1^\kappa} \quad\quad\text{and}\quad\quad
\frac{1}{ q}=\frac{1-\theta}{ q_0}+\frac{\theta  }{ q_1}\, , 
\end{equation*}
the multilinear operator 
\[
T_\theta: L^{p^1}(X_1) \times \cdots \times L^{p^m}(X_m) \to L^q(X_0)
\]
is bounded with an appropriate estimate on its norm. We refer to Theorem \ref{thm-multilinear}  for the full statement. The reader easily recognizes the resemblance to the Stein's complex interpolation in the linear context. 

In the multilinear setting, this type of results already appeared in the work  of  Grafakos and   
Masty\l{}o \cite{GrafakosMastylo}.  The theorem in \cite{GrafakosMastylo} was proved   in the  more general setting of quasi-Banach spaces. However the analyticity and admissibility required  there  were  in the pointwise sense. The admissibility used there is that for every 
$(f^1,\dots, f^m) \in L^{p^1}(X_1)\times \cdots  \times L^{p^m}(X_m)$ and a.e. $y \in X_0$, the mapping 
\[ z \mapsto T_z(\varphi_1,\dots, \varphi_m)(y) \]
is of admissible growth.  Unlike the integral condition \eqref{eq01}, the pointwise admissibility is not easy to check when the operators are not  explicit. 
The extension to more general rearrangement invariant spaces over $X_j$ is not important in our applications and is not pursued here. 

Section \ref{sec3} is devoted  to some bilinear estimates  for Schr\"odinger operators. We consider 
$L = -{\rm div}( A \nabla) + V$, where $A = (a_{kl})_{1\le k, l \le n}$ is a symmetric matrix with real-valued and bounded measurable entries and $V$ is a nonnegative locally integrable potential on $\RR^n$. We prove that for every $p \in (1,\infty)$ and $\alpha, \beta \in [0, \infty)$, there exists a constant $ C(\alpha, \beta, \gamma, p)$, independent of the dimension $n$,  such that
 \begin{equation*}
\int_0^\infty \int_{\RR^n} | {\bf \Gamma} L^\alpha e^{-tL} f (x)\cdot {\bf \Gamma} L^\beta e^{-tL}g(x) |\,  dx\, t^{\alpha + \beta} dt\le C(\alpha, \beta, \gamma, p)  \| f \|_{L^p} \| g \|_{L^{p'}}.
\end{equation*}
  Here ${\bf \Gamma}$ is either $\nabla$ or multiplication by $\sqrt{V}$. The result for $\alpha = \beta = 0$ is due to  Dragicevic and   Volberg \cite{DraVolb1}. Our  proof relies heavily on their result as well as the  interpolation theorem applied to an appropriate analytic family of bilinear operators. 

Finally, we provide the reader with some useful results on log-subharmonic functions  in the Section  \ref{sec4} (Appendix).\\

\noindent{\bf Acknowledgements.} The research of L. Grafakos is partially supported by the  Simons  Foundation Grant  624733 and by the Simons Fellows award  819503. 
 The research of E.~M. Ouhabaz is partly supported by the ANR project RAGE,  ANR-18-CE-0012-01. 

\section{Some preliminary facts}\label{sec2}

Throughout this section, 
  $(X,d,\mu)$ will be a metric space equipped with  a metric $d$ and with a positive measure $\mu$ on a $\si$-algebra $\mathcal A$ of subsets of $X$. 
Let $x\in X$ and $r>0$. A ball  $B(x,r)$  is the set of points $
B(x,r)= \{y\in X:\,\, d(x,y)<r\}.$
We assume   the following mild assumptions   on $(X,d,\mu)$:

\begin{enumerate}
\item[(i)] $\mu(B(x,r))<\nf$ for any $x\in X$ and $r>0$,
\item[(ii)] $\mu$ is a regular measure with respect to the topology of $X$, i.e., for any $A\in \mathcal A$ with $\mu(A)<\nf$ one has
\begin{eqnarray*}
\mu(A) & =&  \sup\{\mu(K):\,\, \textup{$K$ is compact subset of $A$} \big\} \\
\mu(A) & =&  \inf\{\mu(U):\,\, \textup{$U$ is open subset of $X$ and contains $A$} \big\}.  
\end{eqnarray*}
\end{enumerate}



Simple functions on $X$ have the form:
$\sum_{j=1}^N \la_j \chi_{A_j},$
where $\la_j$ are complex numbers and $A_j\in \mathcal A$ are pairwise disjoint 
and satisfy $\mu(A_j)<\nf$. 
Simple functions are dense in $L^p(X)$ for any 
$0<p<\nf$ (as   $(X,\mu)$ is   $\si$-finite).  Moreover, as the sets $A_j$ can be approximated 
from below by compact sets, simple functions with $A_j$  being   
  compact sets are dense in  $L^p(X)$ for $0<p<\nf$.  
We denote by $L^1_{\textup{loc}}(X)$ the space of all measurable functions on $X$ that are integrable over any compact subset of $X$. We also denote by $\mathcal C_c(X)$  the space of all continuous functions with compact support in $X$. 
The subsequent lemma   guarantees the abundance of such functions. 

\begin{lemma}\label{lem-U}
Given $K$ compact and $U$ open subsets of $X$ such that $K\subset U$, there exists  $h\in \mathcal C_c(X)$ such that 
\[
\chi_{K} \le h \le \chi_{ U} .
\]
\end{lemma} 

\begin{proof} 
The sets $K$ and $X\setminus U$ are closed and disjoint,  
so by Urysohn's lemma (which is applicable on metric spaces) 
there is a continuous function $h:X\to [0,1]$ 
that is equal to $1$ on 
$K$ and $0$ on $X\setminus U$.  This function $h$ satisfies the claim. 
\end{proof} 

The following result, inspired by 
\cite{CT}, allows us to approximate ${\mathcal C}_c (X)$ functions by 
   functions that are analytic in a new auxiliary variable $z$. 

\begin{lemma}\label{lem:016}
  Let    $0< p_0 \le p_1   \le    \infty$ satisfy  $p_0<\nf$,  and 
    define $p$ via $1/p=(1-\theta)/p_0+\theta/p_1$, where $0 <\theta < 1$.  
  Given $f\in {\mathcal C}_c (X)$ and   $\ve>0,$ there exist $N_\ve$ and 
  $h_j^\ve\in {\mathcal C}_c (X)$  supported in pairwise disjoint open sets
  $U_j^\ve$, $j=1,\dots, N_\ve$,   and there exist nonzero complex constants $c_j^\ve$ such that  the functions
\begin{equation}\label{EQ-Fz}
  f_{z}^{\ve}  =  \sum_{j=1}^{N_\ve} |c_j^\ve|^{\frac p{p_0} (1-z) +  
 \frac p{p_1}  z}  \, h_j^\ve
\end{equation}
satisfy 
\begin{equation}\label{newest}
\big\|f_{\theta}^{\ve}-f\big\|_{L^ {p_0}} \le  \ve, \q \qq
\begin{cases} 
  \big\|f_{\theta}^{\ve}-f\big\|_{ L^{p_1}}  \le  \ve &\q\textup{if $p_1<\nf$} \\  & \\
    \big\|f_{\theta}^{\ve} \big\|_{L^ \nf} \le \big\| f\big\|_{L^\nf} +\ve 
&\q\textup{if $p_1=\nf$,}  
\end{cases} 
\end{equation} 
and
\begin{equation}\label{newest2}
  \big\|f_{it}^{\ve}\big\|_{L^{p_0}}^{p_0} \le   \big\|f \big\|_{L^p}^p +\ve'  \, , \q
  \big\|f_{1+it}^{\ve}\big\|_{L^{p_1}} \le \big(   \big\|f \big\|_{L^p}^p 
 +\ve'\big)^{\f{1}{p_1}} \, ,
\end{equation}
where $\ve'$ depends on $\ve ,p, \|f\|_{L^p}$ and tends to zero as $\ve\to 0$.
\end{lemma}

\begin{proof} 
Given $f\in {\mathcal C}_c(X)$  and   $ \ve>0$, let $E= \textup{supp} f 
$. 
Let $E'=\bigcup_{x\in E} B(x,1)$ and notice that in view of the compactness of $E$, the set 
$E'$ has finite measure. By the uniform continuity of $f$ there is a $\de 
$ in $(0,1)$  
such that 
\[
x,y\in X,\,\,\,\, d(x,y)<\de\implies |f(x)-f(y)|<
 \f{ \ve}{2^{ \max ( 1, \f 1{ p_0}  ) } }
  \bigg( \f{1 }{  1+  \mu(E')     } \bigg)^{    \f 1{ p_0}     } .
\]
Then we   cover the support of $f$ by finitely many balls
$B_1, \dots , B_{N_\ve'}$ of radius $\de/2$. We   find pairwise disjoint
measurable subsets $A_j$ of $B_j$ 
that   satisfy $B_1\cup \cdots \cup B_{N_\ve'}=
A_1\cup \cdots \cup A_{N_\ve}$; notice that this union contains $E$ and is contained in $E'$. 
 Suppose that  $N_\ve $ of the $A_j$ are nonempty, without loss of generality 
 assume these are the first $N_\ve$; this way we have $A_j\neq \emptyset$ 
for all $j\le N_\ve\le N_\ve'$. We now let 
$c_j^\ve = f(x_j)$, where $x_j$ is any fixed point in $A_j$.  As a 
consequence of these choices one has
\[
\big\| f- \sum_{j=1}^{N_\ve} c_j^\ve \chi_{A_j}\big\|_{L^{\nf}} \le  
 \f{ \ve}{2^{ \max ( 1, \f 1{ p_0}  ) } } \bigg( \f{1 }{  1+  \mu(E')     
} \bigg)^{    \f 1{ p_0}     }. 
\]
It follows from this that if $p_1=\nf$ then 
\begin{equation}\label{7784}
\sup_{1\le j\le N_\ve} |c_j^\ve| = 
\big\|   \sum_{j=1}^{N_\ve} c_j^\ve \chi_{A_j}\big\|_{L^{\nf}} \le  \big\| f\big\|_{\li}+ \ve ,
\end{equation}
while 
if  $p_1<\nf$ then
\begin{eqnarray*}
\big\| f- \sum_{j=1}^{N_\ve} c_j^\ve \chi_{A_j}\big\|_{L^{p_\kappa}}^{\min(1,{p_\kappa})}
   \le \Big[  \f{ \ve}{2^{ \max ( 1, \f 1{ p_0}  ) } } \bigg( \f{1 }{  1+  \mu(E')     } \bigg)^{    \f 1{ p_0}     }   
 \mu \Big(\bigcup_{j=1}^{N_\ve} A_j\Big)^{\f {1}{p_\kappa}} \Big]^{\min(1,{p_\kappa})}  
  \le   \f{\ve^{\min(1,{p_\kappa})} }{2}  ,
\end{eqnarray*}
where $p_\kappa \in\{p_0,p_1,p\}$.

By the regularity of $\mu$ we pick compact sets $K_j$ contained in $A_j$ such that $\mu(A_j\setminus K_j)<\f{\eta}2$, 
for some $\eta>0$ chosen to satisfy
\[
\max_{\kappa\in \{0,1\}} 
\bigg( \sum_{j=1}^{N_\ve} |2c_j^\ve|^{\min (1,p_\kappa)} \bigg) \eta^{ \min (1,\f{1}{ p_\kappa} )} <
 \f{\ve^{\min(1,{p_\kappa})} }{2}  .
\]
Then the compact sets $K_j$ are pairwise disjoint, so 
\[
\min_{j \neq k} \big( \textup{dist} (K_j, K_k) \big) = \rho>0.
\]
Now let 
\[
U_j'= \bigcup_{x\in K_j} B\big(x, \f{\rho}{3}\big) 
\]
and choose $U_j''$ open such that $A_j\subset U_j''$ and $\mu (U_j''\setminus A_j)<\f{\eta}2$ 
by the regularity of $\mu$. 
Then define 
\[
\hspace{1in}U_j^{\ve}= U_j'\cap U_j'', \qq\qq j=1,2,\dots, N_\ve.
\]
The sets $U_j^{\ve}$ are open and pairwise disjoint. Also each $U_j^\ve$ contains the compact set $K_j$. 
 By Lemma \ref{lem-U} we
  pick $g_j^\ve\in \mathcal C_c(X)$ with values in $[0,1]$ 
  satisfying  $\chi_{K_j} \le g_j^\ve\le \chi_{U_j^{\ve}}$. Then if $p_1<\nf$ 
by  the subadditivity of     
$\| \cdot \|_{p}^{\min (1,p)} $ we write 
\begin{eqnarray*}
\Big\| f - \sum_{j=1}^{N_\ve} c_j^\ve g_j^\ve \Big\|_{L^{p_\kappa }}^{\min (1,p_\kappa)}  
 &\le& 
\Big\| f - \sum_{j=1}^{N_\ve} c_j^\ve \chi_{A_j}  \Big\|_{L^{p_\kappa }}^{\min (1,p_\kappa)}  
+
\Big\|   \sum_{j=1}^{N_\ve} c_j^\ve ( \chi_{A_j} -g_j^\ve )\Big\|_{L^{p_\kappa }}^{\min (1,p_\kappa)}  \\
 & \le &  \f{  \ve^{\min (1,p_\kappa)}}2   + \sum_{j=1}^{N_\ve} 
 |2c_j^\ve|^{\min (1,p_\kappa)}   \eta^{ \min (1,\f{1}{ p_\kappa} )} \\
  & \le &\ve^{\min (1,p_\kappa)}  ,
\end{eqnarray*}
as the   $\chi_{A_j} -g_j^\ve $ is bounded by $2$ and supported in $U_j^{\ve}\setminus K_j$ 
which has measure at most $\eta$.  
This proves \eqref{newest} when $p_1<\nf$. 
Note that the same argument shows that 
\begin{equation}\label{samearg3}
\Big\| f - \sum_{j=1}^{N_\ve} c_j^\ve \chi_{U_j^{\ve}} \Big\|_{L^{p_\kappa }} \le \ve , 
\qq \kappa\in \{0,1\}. 
\end{equation}
We set $h_j^\ve = e^{i\phi_j^\ve} g_j^\ve$, where $\phi_j^\ve$ is the argument of the complex number $c_j^\ve$. 
Then $h_j^\ve$ is that function claimed in  \eqref{EQ-Fz}. 
Observe that  
\[  
f_{\theta}^{\ve} = \sum_{j=1}^{N_\ve} |c_j^\ve| h_j^\ve =  \sum_{j=1}^{N_\ve}  c_j^\ve  g_j^\ve 
\]
satisfies   \eqref{newest} when $p_1<\nf$; in the case $p_1=\nf$ we have
\[
|f_{\theta}^{ \ve}|\le \sum_{j=1}^{N_\ve} |c_j^\ve| \chi_{U_j^{\ve} }   
\le \sup_{j} |c_j^\ve| \le  \big\|f \big\|_{L^\nf}  +\ve
\]
by \eqref{7784}. 
Thus \eqref{newest} holds when $p_1=\nf$. We now write 
\[
\big\| f_{it}^{\ve} \big\|_{L^{p_0}}^{p_0} \le  \sum_{j=1}^{N_\ve}  | c_j^\ve |^p \mu( U_j^{\ve}  ) 
= \bigg\|\sum_{j=1}^{N_\ve} c_j^\ve \chi_{U_j^{\ve}} \bigg\|_{L^p}^p
\le \Big( \ve^{\min (1,p) } +\big\| f\big\|_{L^p}^{\min (1,p) }   \Big)^{\f p {\min (1,p)}} ,
\]
having   used \eqref{samearg3}.  

We set $\ve'=\ve^p$ if $p\le 1$ and $\ve'=(\ve+\|f\|_{L^p})^p - \|f\|_{L^p}^p$ when $1<p<\nf$. 
Then $\ve'\to 0$ as $\ve\to 0$ and this
 proves \eqref{newest2}  for $p_0$ and analogously for $p_1$  when $p_1<\nf$;  now if $p_1=\nf$ then  
$\| f_{1+it}^{\ve} \|_{L^{\infty}}\le 1$ and the right hand side of the second 
inequality in \eqref{newest2}  is equal to $1$, so the inequality is still valid.  
 \end{proof}

 Throughout this paper  ${\mathbf S} $ will denote  the  open unit strip  ${\mathbf S}=\{z\in \CC:\,\, 0<\real z<1\} $ and 
  $ \overline{{\mathbf S}}$ its closure, i.e., the closed unit strip. 
  As the boundary of $\mathbf S$ has two disjoint pieces and integration over each piece will be written 
separately, we  introduce the ``half''  Poisson kernel $\Om$ on 
$\overline{\mathbf S}\setminus\{1\}$ via:
\bee\label{PSonjstrip}
 \Omega(x,y) = \dfrac12\, \dfrac{\sin (\pi x)}{\cosh (\pi y )+ \cos (\pi x)} 
\eee
where $0\le x \le 1$ and $-\nf <y<\nf$ but $(x,y)\neq (1,0)$.  This function is nonnegative  
and satisfies 
\begin{equation}\label{FPx}
\int_{-\nf}^{+\nf} \Om(x,t) \, dt = x \qq\textup{for all $0\le x<1$.}
\end{equation}
The next result due to Hirschman~\cite[Lemma 1]{Hirschman} 
is fundamental in complex interpolation.

\begin{proposition}\label{3linesHirschman}
  Let $F $ be a  continuous  function on the closed unit strip $\overline{\mathbf S}$ such that 
  $\log |F|$    is  subharmonic  in $\mathbf S$  that   satisfies 
\bee\lab{claimedineq8887774ii}
\sup_{0\le x\le 1}   \log   |F(x+iy)|  \le C \, e^{a \, |y|} ,\qq    -\nf 
<y<\nf,
\eee
  for some fixed $C,a>0$ with $ a <\pi$. If $N_0$, $N_1$ are   continuous 
functions on the line 
  that satisfy $   N_0(y) \ge \log |F(iy)|$ and 
  $\   N_1(y) \ge \log |F(1+iy)|$ for all $y\in (-\nf, \nf)$,   then for any $\theta \in (0,1)$ we have
\bee\lab{claimedineq8887774}
\log |F(\theta)|  \le  \int_{-\nf}^{+\nf} \Om (1-\theta ,t) N_0(t) \, dt + \int_{-\nf}^{+\nf} \Om ( \theta, t) N_1(t) \, dt .
\eee
 \end{proposition}

\section{Interpolation for analytic families of multilinear operators} 

 Throughout this section
 $(X_j,d_j,\mu_j)$, $0\le j \le m$,     are metric measure spaces
that satisfy assumptions (i) and (ii). 

\begin{definition}\label{DefAFAG}
Suppose
that for every $z \in \overline{{\mathbf S}}$
there is  an associated  $m$-linear
operator $T_z$ defined on ${\mathcal C}_c(X_1)\times \cdots \times {\mathcal C}_c(X_m)$
and taking values in $ L^1_{\textup{loc}} (X_0)$.  
We call $\{T_z\}_z$  an 
\emph{analytic family}\index{analytic family of operators} 
  if for all $(\vp_1,\dots, \vp_m ) $ in $\mathcal C_c(X_1)\times \cdots \times 
 C_c(X_m)$ and $w$ bounded  
 function 
with compact support on $X_0$   the mapping   
\begin{equation}\label{eqn: 1.3.analy}
z\mapsto  \int_{X_0}  T_z(\vp_1,\dots, \vp_m)\,  w \,d\mu_0
\end{equation}
is analytic in the open strip
${\mathbf S}$ and
continuous on its closure. The analytic family $\{T_z\}_z$ is called of
\emph{admissible growth}\index{admissible growth} if there is a constant $\ga$ with
$0\le \ga<\pi$ and an $s$ satisfying   $1\le  s \le \nf$, 
 such that for any  $(\vp_1,\dots , \vp_m)$ in $\mathcal C_c(X_1)\times \cdots \times 
 C_c(X_m)$ and  
 every compact subset $K$ of $X_0$ 
 there is   a constant $ C(\vp_1,\dots, \vp_m,K )$ such that
\begin{equation}\label{eqn: 1.3.admis.growth}
\log \bigg[ \int_K | T_z(\vp_1,\dots, \vp_m )|^sd\mu_0 \bigg]^{1/s}  \le C(\vp_1,\dots, \vp_m,K )\, e^{ \ga |\textrm{Im}\,z |}, 
\q \textup{for all $z\in \overline{\mathbf S}.$}
\end{equation}
\end{definition}

Now we state the main result on interpolation of analytic  multilinear operators. 

\begin{theorem}\label{thm-multilinear}
For $ z\in \overline{\mathbf S}$,  let   $T_z$  be an $m$-linear operator 
     on ${\mathcal C}_c(X_1)\times \cdots \times {\mathcal C}_c(X_m)$
with  values  in $ L^1_{\textup{loc}} (X_0)$ that form an analytic family  
 of admissible growth.  For $\kappa\in \{1,\dots , m\}$
let $0 < p_0^\kappa ,  p_1^\kappa \le \infty$,   $0 < q_0,  q_1 \le \infty$, 
fix $0< \theta < 1$,  and define $p^\kappa,q$ by the equations
\begin{equation}\lab{1.3.indices-10-00}
\frac{1}{ p^\kappa}=\frac{1-\theta}{ p_0^\kappa}+
\frac{\theta  }{p_1^\kappa} \quad\quad\text{and}\quad\quad
\frac{1}{ q}=\frac{1-\theta}{ q_0}+\frac{\theta  }{ q_1}\, .
\end{equation}
Suppose that for all  $(f^1,\dots, f^m)\in {\mathcal C}_c(X_1)\times \cdots \times {\mathcal C}_c(X_m)$   we have
\begin{align}
\big\|T_{iy}(f^1,\dots, f^m)\big\|_{L^{q_0} (X_0) }   \le &\,\, B_0M_0(y )
\prod_{\kappa=1}^m \|f^\kappa\|_{L^{p_0^\kappa} (X_\kappa) }\, , \label{eqn: 3.1.372}\\
\big\|T_{1+iy}(f^1,\dots, f^m)\big\|_{L^{q_1} (X_0)}  
 \le &\,\, B_1 M_1(y ) \prod_{\kappa=1}^m \|f^\kappa  \|_{L^{p_1^\kappa} (X_\kappa)} \, , \label{eqn: 3.1.373}
\end{align}
where $M_0$ and $M_1$ are nonnegative continuous functions on the real 
 line that  satisfy  
 \begin{equation}\label{eqn:1.3.hfytr}
M_0(y)   \le  e^{ c \, e^{\tau  |y|}}, \qq M_1(y)   \le   e^{ c \, e^{\tau  |y|}}
\end{equation}
for some $ c,\tau  \ge 0$ with $\tau <\pi$, and $B_0,B_1>0$.  
Then for all   $f^j$ in ${\mathcal C}_c (X_j) $, $1\le j\le m$,     we have
\begin{equation}\label{eqn: 3.1.876543}
\big\|T_{\theta } (f^1,\dots, f^m)\big\|_{L^{q} (X_0) } \le  B_0^{1-\theta} B_1^{\theta}  M(\theta ) 
 \prod_{\kappa=1}^m \big\|f^\kappa \big\|_{L^{p^\kappa} (X_\kappa) }, 
\end{equation}
where 
\begin{equation*}
M(\theta) =\exp \bigg\{ 
\int_{-\infty}^{\infty} \Big[
 \Om(1-\theta, y) \log M_0(y)  + \Om(\theta, y) \log M_1(y) \Big]\,\, dy \bigg\} .
\end{equation*}
\end{theorem}

\begin{proof}   
\noindent {\bf  Case I}: $ \min (q_0,  q_1) > 1$.  

This assumption forces $ q_0',q_1'<\nf$ and so $q'<\nf$ as well. 
Given $T_z$ as in the statement of the theorem, for   
$f^j\in \mathcal C_c(X_j)$, $1\le j \le m$,  and $  g\in \mathcal C_c(X_0)$ one may 
be tempted to 
consider the   family of operators
\[
H(z)= \int_{X_0} T_z(f^1,\dots , f^m)\, g \, d\mu_0
\]
 which is analytic in $\mathbf S$, continuous and bounded in $\overline{\mathbf S}$ 
and satisfies the hypotheses of Proposition \ref{3linesHirschman} with  bounds 
\[
|H(iy)|\le B_0 \, M_0(y)  \prod_{\kappa=1}^m\| f^\kappa \|_{L^{p_0^\kappa}}    
\|g \|_{L^{q_0'}} 
,\q |H(1+iy)|\le B_1 \, M_1(y)  \prod_{\kappa=1}^m \| f^\kappa \|_{L^{p_1^\kappa}} 
   \|g \|_{L^{q_1'}}
\]
 for all real $y $.  
Applying the result of    Proposition \ref{3linesHirschman} and identity  
 \eqref{FPx} (with $x=1-\theta$ and $x=\theta$) yields  for all $f^1,\dots, f^m\in \mathcal C_c(X)$ 
and $g\in \mathcal C_c(X_0)$ 
\begin{equation}\begin{split} \label{aux94948}
\bigg| \!  \int_{X_0} \!\!\!  T_\theta(f^1,\dots, f^m)\, g\, d\mu_0 \bigg|  
  \le M(\theta) 
\big( B_0 \prod_{\kappa=1}^m  \| f^\kappa 
 \|_{L^{p_0}}  \|g \|_{L^{q_0'}}\big) ^{1-\theta} 
\big( B_1 \prod_{\kappa=1}^m  \| f^\kappa  \|_{L^{p_1}} 
   \| g \|_{L^{q_1'}}\big) ^{ \theta}  .
\end{split} \end{equation}  
Unfortunately  this bound does not provide the claimed assertion; it supplies, however, a useful 
continuity estimate for the operator   $T_\theta$.

To improve \eqref{aux94948}, let us first consider the situation where $
\min ( p_0 ^\kappa, p_1^\kappa) <\nf$ for some $\kappa \in \{1,\dots, m\}$, which 
forces $p^\kappa<\nf$ for the same $\kappa$.
  Fix $f^j\in\mathcal C_c(X_j)$, $g\in \mathcal C_c(X_0)$
    and $\ve>0$. By Lemma \ref{lem:016} we can 
find   $f_z^{1,\ve} ,\dots , f_z^{m,\ve} $ and $ g_z^\ve$ such that 
  \[
  f_z^{\kappa, \ve}=  \sum_{j_\kappa=1}^{N_\ve^\kappa}|c_{j_\kappa} ^{\kappa,\ve} |^{\frac {p^\kappa}{p_0^\kappa} (1-z) +   \frac {p^\kappa}{p_1^\kappa}  z} u_{j_\kappa}^{\kappa,\ve}, 1\le \kappa\le m, 
  \q 
  g_z^\ve=  \sum_{k=1}^{M_\ve} |d_k^\ve|^{\frac {q'}{q_0'} (1-z) +   \frac {q'}{q_1'}  z} v_k^\ve, 
  \]
  where $(u_{j_1}^{1,\ve},\dots, u_{j_m}^{m,\ve})$ lies in $\mathcal C_c(X_1)\times 
  \cdots \times \mathcal C_c(X_m)$,  
    $  v_k^\ve$ in $\mathcal C_c(X_0)$,  and 
  \begin{equation}
   \| f_\theta^{\kappa,\ve} -f^\kappa  \|_{L^{p_0}}<\ve , \   \|{g_\theta^\ve-g}\|_{L^{q_0'}}<\ve, \ 
 \| f_\theta^{\kappa,\ve}-f^\kappa\|_{L^{p_1}}<\ve , \   
 \|{g_\theta^\ve-g}\|_{L^{q_1'}}<\ve \label{eqn: 1.3.86556}  
 \end{equation}
 \begin{equation}
  \|{f_{it}^{\kappa, \ve}}\|_{L^{p_0}}\le   \big( \| f^\kappa \|_{L^{p}}+\ve' \big)^{\frac p{p_0}} ,\quad
 \|{g_{it}^\ve}\|_{L^{q_0'}}\le   \big( \| g \|_{L^{q'}}+\ve' \big)^{\frac {q'}{q_0'}} , \label{eqn: 1.3.87} 
 \end{equation}
 \begin{equation}
   \|{f_{1+it}^{\kappa,\ve}}\|_{L^{p_1}}\le   \big( \|{f^\kappa}\|_{L^{p}}+\ve' \big)^{\frac p{p_1}},\quad
  \|{g_{1+it}^\ve}\|_{L^{q_1'}}\le  \big( \|{g}\|_{{q'}}+\ve' \big)^{\frac {q'}{q_1'}}, \label{eqn: 1.3.88}
  \end{equation}
 where  $\| f_\theta^{\kappa,\ve}-f^\kappa \|_{L^{p_1^\kappa}}<\ve$ in \eqref{eqn: 1.3.86556} 
 is replaced by $ \|f_{\theta}^{\kappa,\ve}  \|_{L^\nf} \le  \| f^\kappa \|_{L^\nf}+\ve$, if $p_1^\kappa=\nf$ and analogously   if $p_0^\kappa=\nf$. 
 
Now consider the function defined on   the closure of the unit  strip
\[
F(z) = \int_{X_0} T_z(f_z^{1, \ve}, \dots , f_z^{m, \ve} ) \, g_z^\ve\, d\mu_0 
\]
\[
= \sum_{ \substack{ 1\le j_1 \le  N^{1}_\ve \\ \vdots \\  
1\le j_m \le  N^{m}_\ve }}  \sum_{k=1}^{M_\ve} \Bigg\{ \bigg[ \prod_{\kappa=1}^m
 |c_{j_\kappa}^{\kappa,\ve}    
 |^{\frac {p^\kappa} {p_0^\kappa} (1-z) +   \frac {p^\kappa}{p_1^\kappa}  
z} 
  |d_k^\ve|^{\frac {q'}{q_0'} (1-z) +   \frac {q'}{q_1'}  z} \bigg] \int_{X_0} T_z(u_{j_1}^{1,\ve},\dots, 
  u_{j_m}^{m,\ve}) v_k^\ve\, d\mu_0  \Bigg\}.
\]
  
Applying H\"older's inequality with exponents $s$ and $s'$ to 
$ \int_{X_0} T_z(u_{j_1}^{1,\ve},\dots, 
  u_{j_m}^{m,\ve}) v_k^\ve\, d\mu_0 $
and using condition \eqref{eqn: 1.3.admis.growth}
 we obtain for any $z$ in $\overline{\mathbf S}$  
\[
 |F(z) | \le  \Bigg[\prod_{\kappa=1}^m \sum_{j_\kappa=1  } ^{   N^{\kappa}_  \ve } 
 |c_{j_\kappa}^{\kappa,\ve} 
 |^{\frac {p^\kappa} {p_0^\kappa} +   \frac {p^\kappa}{p_1^\kappa}   } 
  \sum_{k=1}^{M_\ve}
  |d_k^\ve|^{\frac {q'}{q_0'}  +   \frac {q'}{q_1'}  }   \| v^\ve_k\|_{L^{s'}}\Bigg] 
  e^{ [\max\limits_{j_1,\dots , j_m,k} C(u_{j_1}^{1,\ve} , \dots , u_{j_m}^{m,\ve}, \textup{supp } v_k^\ve)  ]  e^{\ga  |\imag z|}   } 
  \]
  \[
  \le e^{C'  e^{\ga  |\imag z|}   }  ,
\] 
where $C'$ equals $\max_{j_1,\dots , j_m,k} C(u_{j_1}^{1,\ve} , \dots , u_{j_m}^{m,\ve}, \textup{supp } v_k^\ve)$ plus the logarithm of 
the double sum in the square brackets. 
Thus  $F$ satisfies the hypothesis of Proposition~\ref{3linesHirschman}, as $\ga<\pi$. 
 
H\"older's inequality,    hypothesis
(\ref{eqn: 3.1.372}) and (\ref{eqn: 1.3.87})     give for   $y$ real
\[
|F(iy)| \le  B_{0}M_0(y)
\prod_{\kappa=1}^m
\big\|f_{iy}^{\kappa,\ve}\big\|_{L^{p_0^\kappa}} \big\|g_{iy}^\ve\big\|_{L^{q_0'}} 
\le   B_{0}M_0 (y) \prod_{\kappa=1}^m
\big( \|f^\kappa\|_{L^{p}}+\ve'\big)^{\frac {p^\kappa}{p_0^\kappa}}
\big( \|{g}\|_{L^{q'}}+\ve'\big)^{\frac {q'}{q_0'}}.  
\]

Likewise,  H\"older's inequality,  the hypothesis
(\ref{eqn: 3.1.373}) and (\ref{eqn: 1.3.88}) imply  for   $y$ real
 \begin{eqnarray*}
|F(1+iy)|  
&\le& B_{1}M_1(y) \prod_{\kappa=1}^m
\big\|f_{1+iy}^{\kappa,\ve}\big\|_{L^{p_1^\kappa}} \big\|g_{1+iy}^\ve\big\|_{L^{q_1'} }\\
& \le& 
  B_{1}M_1 (y) \prod_{\kappa=1}^m
  \big( \|f^\kappa\|_{L^{p^\kappa}}+\ve'\big)^{\frac {p^\kappa} {p_1^\kappa}}
  \big( \|{g}\|_{L^{q'}}+\ve'\big)^{\frac {q'}{q_1'}}.  
\end{eqnarray*}

As $\log |F|$ is subharmonic in ${\mathbf S}$, applying Proposition~\ref{3linesHirschman} we 
 obtain
\[
\log |F(\theta)|  \le  \int_{-\nf}^{+\nf} \Om (1-\theta ,t) \log [ M_0(t)\,Q_0] \, dt + 
\int_{-\nf}^{+\nf} \Om ( \theta, t) \log [M_1(t)\,Q_1] \, dt  , 
\]
where $\Om$ is  the Poisson kernel on the strip [defined in \eqref{PSonjstrip}] and 
\[
Q_0= B_0 \, \prod_{\kappa=1}^m\big( \|f^\kappa \|_{L^{p}}+\ve'\big)^{\frac {p^\kappa} {p_0^\kappa}}\big( \|{g}\|_{L^{q'}}+\ve'\big)^{\frac {q'}{q_0'}} , \,\,\,\,\,
Q_1=  B_1 \,  \prod_{\kappa=1}^m \big( \|f^\kappa\|_{L^{p}}+\ve'\big)^{\frac {p^\kappa}{p_1^\kappa}}\big( \|{g}\|_{L^{q'}}+\ve'\big)^{\frac {q'}{q_1'}} .
\]

Using identity \eqref{FPx} (with $x=1-\theta$ and $x=\theta$) and the 
fact  that 
\[
Q_0^{1-\theta}  Q_1^\theta =  B_0^{1-\theta} B_1^{\theta}  
 \prod_{\kappa=1}^m\big( \|f^\kappa \|_{L^{p}}+\ve'\big) \big( \|{g}\|_{L^{q'}}+\ve'\big) 
\]
we obtain    (with $M(\theta) $ as in the statement of the theorem) that
\begin{equation}\label{11AI-F}
\bigg|  \int_{X_0}\!\!\!  T_\theta(  f_\theta^{1,\ve}, \dots , f_\theta^{m,\ve} ) g_\theta^\ve \, d\mu_0 \bigg| = |F(\theta) | \le  
  M(\theta)  B_0^{1-\theta} B_1^{\theta} 
 \prod_{\kappa=1}^m\big( \|f^\kappa \|_{L^{p}}+\ve'\big) \big( \|{g}\|_{L^{q'}}+\ve'\big) .
\end{equation}
An application of the triangle inequality gives 
\begin{eqnarray}
&&\bigg|  \int_{X_0} T_\theta(  f^1,\dots, f^m) \, g  \, d\mu_0 - \int_{X_0} T_\theta(  f_\theta^{1,\ve}, \dots , f_\theta^{m,\ve}  )\, g_\theta^\ve \, d\mu_0  
\bigg|    \lab{5asdjhtteeq}  \\
&&\q   \le \sum_{\kappa=1}^m \bigg|  \int_{X_0} T_\theta( f^1_\theta,\dots, f^{\kappa-1}_\theta, f^\kappa - 
f^\kappa_\theta, f^{\kappa+1},\dots, f^m ) \,  g   \, d\mu_0 
\bigg| \notag \\
&& \qq\qq
+ \bigg|  \int_{X_0} T_\theta(  f_\theta^{1,\ve}, \dots , f_\theta^{m,\ve} ) \, (g-g_\theta^\ve  )\, d\mu_0   
\bigg| . \notag
\end{eqnarray}
We now apply \eqref{aux94948} in each of the terms on the right side of the 
inequality and we use
\eqref{eqn: 1.3.86556}.
We deduce that \eqref{5asdjhtteeq} tends to zero as $\ve\to 0$. 
We conclude 
\begin{equation}\label{11AI-Fee}
\bigg|  \int_{X_0} T_\theta(  f^1,\dots, f^m  )\, g  \, d\mu_0 \bigg|   \le  
  M(\theta)  B_0^{1-\theta} B_1^{\theta} 
\prod_{\kappa=1}^m   \|f^\kappa\|_{L^{p}}  \|{g}\|_{L^{q'}}  .
\end{equation}
Finally we obtain \eqref{eqn: 3.1.876543} by   taking the supremum  
in \eqref{11AI-Fee} over all 
$g$ in $\mathcal C_c(X_0)$ with $L^{q'}$ norm equal to $1$. 

Suppose now that $ p_0^\kappa=p_1^\kappa =\nf$ for some $\kappa$. This 
forces  $p^\kappa=\nf$ for these $\kappa$. Without loss of generality assume that 
$ p_0^\kappa=p_1^\kappa =\nf$ for all $\kappa \le \la$ and $\min ( p_0^\kappa, 
 p_1^\kappa)<\nf$ for all $\kappa\in \{\la+1,\dots , m\}$. 
We repeat the preceding argument working with the analytic function
\[
F(z) = \int_{X_0} T_z(f^1,\dots, f^\la, f_z^{\la+1,\ve} , \dots, f_z^{m,\ve} ) \, g_z^\ve \, d\mu_0   
\]
on  ${\mathbf S}$ which is multilinear of a lower degree and satisfies the initial estimates
\[
\big| F(iy) \big|  \le B_0  
\Big( \prod_{\kappa=1}^\la  \big\|f^\kappa\big\|_{L^\nf} \Big)   M_0(y) 
\Big(\prod_{\kappa=\la+1}^m  \big\|f^\kappa\big\|_{L^{p^\kappa_0} } \Big)
\big\|g\big\|_{L^{q'}} .
\]
and 
\[
\big| F(1+iy) \big|  \le B_1   
\Big( \prod_{\kappa=1}^\la  \big\|f^\kappa\big\|_{L^\nf} \Big)   M_1(y)
\Big(\prod_{\kappa=\la+1}^m  \big\|f^\kappa\big\|_{L^{p^\kappa_1} } \Big)
\big\|g\big\|_{L^{q'}} .
\]
The argument in the previous case using  Proposition~\ref{3linesHirschman}  yields 
\[
\bigg| \int_{X_0} T_\theta (f^1,\dots, f^m)\, g \, d\mu_0  \bigg|  \le B_0^{1-\theta} B_1^{\theta}  
\Big( \prod_{\kappa=1}^\la  \big\|f^\kappa\big\|_{L^\nf} \Big)M(\theta)
\Big(\prod_{\kappa=\la+1}^m  \big\|f^\kappa\big\|_{L^{p^\kappa} } \Big)
\big\|g\big\|_{L^{q'}} .
\]
Finally we  take  the supremum of the integrals  over all 
$g$ in $\mathcal C_c(X)$ with $L^{q'}$ norm equal to $1$, to deduce  
\eqref{eqn: 3.1.876543}.  

\medskip
\noindent {\bf  Case II}:  $ \min (q_0,  q_1) \le  1$.   

 Assume first that $\min (p_0^\kappa,p_1^\kappa)<\nf$ for all $\kappa$.   
Choose    $r>1$ such that $ r\,\min(q_0,q_1)>q.$ 
Let us fix  a nonnegative step function $g$ with $\norm{g}_{L^{r'}(X_0)}=1.$ Assume that
$g = \sum_{k=1}^Ka_k\chi_{E_k},$ where $a_k>0$ and $E_k$ are pairwise disjoint
measurable compact subsets of $X_0$ (hence of finite measure). 
It suffices to work with such dense subsets of $L^{r'}(X_0)$ in view of the 
assumption that $X_0$ is a $\si$-finite metric space. 
For $z\in \CC$ set
\[
  g^z = \sum_{k=1}^Ka_k^{R(z)}\chi_{E_k},
\]
  where we set
\[
 R(z) = r'\Big[1-  \dfrac{q}{rq_0}(1-z) -\dfrac{q}{rq_1} z \Big].
\]
Notice that $R(\theta)=1$. We fix   $f^\kappa\in \mathcal C_c(X)$ and $\ve>0$. Let $f_z^{\kappa, \ve}$ be as in Case I  obtained by Lemma \ref{lem:016}.  Define  the function 
\begin{equation}\label{o34994ngt}
  G(z) = \int_{X_0} \abs{T_{ z}(f_z^{1,\ve} ,\dots, f_z^{m,\ve}) }^{\frac{q}{r}}\abs{g^z }\; d\mu_0 = \sum_{k=1}^K\int_{E_k} \big|   F_k(z,x) \big|^{\frac{q}{r}}\; d\mu_0 (x)  .
\end{equation}
  where
\[
  F_k(x,z)= a_k^{\f rq R(z)} 
  \sum_{ \substack{ 1\le j_1 \le  N^{1}_\ve \\ \vdots \\  
1\le j_m \le  N^{m}_\ve }}   \Bigg[ \prod_{\kappa=1}^m
\bigg(    |c_{j_\kappa}^{\kappa,\ve}|^{\frac {p^\kappa}{p_0^\kappa} (1-z) 
+   \frac {p^\kappa}{p_1^\kappa}  z} \bigg)   T_{ z}(u_{j_1}^{1,\ve}, \dots, u_{j_m}^{m,\ve})(x) \bigg].
\]
  If we knew that each term of the sum on the right in \eqref{o34994ngt} 
  is log-subharmonic, it would follow from 
  Lemma~\ref{logsubh} that so is $G$. To achieve this we use Lemma~\ref{SWTohoku}, which 
  requires knowing that for each $k$, the mapping 
$z\mapsto F_k(\cdot, z)$  is analytic from $\mathbf S$ to $L^1(E_k)$.
To show this, in view of Theorem~\ref{BanachAnalytic}, 
it suffices to show that for any bounded function $w$ supported in $E_k$
  the function $z\mapsto \int_{E_k} F_k(z,x) w(x)\, d\mu_0( x) $ is   analytic   
in $\mathbf  S$ and continuous on its closure;  but this condition is guaranteed
 by the   definition of analytic families.

We plan to apply Proposition~\ref{3linesHirschman} to $G$ and 
we verify its hypotheses. 
Using H\"older's inequality
  with indices $\frac{r q_0}{q} $ and $\big( \frac{r q_0}{q} \big)'$, 
  \eqref{eqn: 3.1.372}, and the fact $\|g\|_{L^{r'}}=1$ we obtain 
  \begin{eqnarray*}
  {G(it)} &\le&  \left\{\int_{X_0}
  \abs{T_{ {it}}(f_{it}^{1,\ve} ,\dots , f_{it}^{m,\ve} ) }^{q_0} d\mu_0 \right\}^{\frac{q}{r q_0}}
  \big\| g^{it } \big\|_{L^{ ( \frac{r q_0}{q}  )'}}  \\  
  &\le&  \, \Bigg[ B_{0}M_0(t)  \prod_{\kappa= 1}^m  \Big(\big\|f^\kappa\big\|_{L^{p^\kappa } }^{p^\kappa}  +\ve'\Big)^{\f 1 {p^\kappa} } \Bigg]^{\frac{q}{r}}.
  \end{eqnarray*}
  Similarly, we obtain the estimate
  \[
  {G(1\!+\! it)}  
  \le  \, \Bigg[ B_{1}M_1(t)  \prod_{\kappa= 1}^m  \Big(\big\|f^\kappa\big\|_{L^{p^\kappa } }^{p^\kappa}  +\ve'\Big)^{\f 1 {p^\kappa} } \Bigg]^{\frac{q}{r}}.
  \]  
  Finally we verify condition   \eqref{claimedineq8887774ii} for $G$. 
Let $E$ be a  compact set  that contains all $E_k$. 
We apply H\"older's inequality
  with indices $\frac{r s}{q} $ and $\big( \frac{r s}{q} \big)'$ to obtain for  $z\in \overline{\mathbf S}$ 
\begin{align*}
 &  {G(z)}   \\
   & \le   \big\|   T_{ {z}}(f_{z}^{\ve}  ) \chi_E \big\|_{L^s } ^{\frac{q}{r }}
  \big\| g^{z } \big\|_{L^{ ( \frac{r s}{q}  )'}}    \\
& \le     \bigg[ \sum_{\substack{1\le j_1\le  N^1_\ve \\ \vdots \\ 1\le j_m\le  N^m_\ve
} }   \prod_{\kappa=1}^m   |c_{j_\kappa}^{\kappa,\ve}|^{\f{p^\kappa}{p_0^\kappa}+\f{p^\kappa}{p_1^\kappa}} 
  \big\|   T_{ z}(u_{j_1}^{1,\ve} ,\dots,  u_{j_m}^{m,\ve})  \big\|_{L^s(E)} \bigg]^{\f qr} 
   \bigg[ \sum_{k=1}^K |d_k|^{r'[ 1+\f qr ( \f 1{q_0} +\f 1 {q_1})] } \big\| 
    \chi_{E_k}  \big\|_{L^{ ( \frac{r s}{q}  )'}}  \bigg]  \\
& \le       e^{\f qr \sup\limits_{j_1,\dots, j_m}\!\!\!\! C(u_{j_1}^{1,\ve},\dots, u_{j_m}^{m,\ve}, E) e^{\ga |\imag z|} }  \bigg[\!\! \sum_{\substack{1\le j_1\le  N^1_\ve \\ \vdots \\ 1\le j_m\le  N^m_\ve
} }   \prod_{\kappa=1}^m   |c_{j_\kappa}^{\kappa,\ve}|^{\f{p^\kappa}{p_0^\kappa}+\f{p^\kappa}{p_1^\kappa}}  \bigg]^{\!\f qr} 
   \bigg[ \sum_{k=1}^K |a_k|^{r'[ 1+\f qr ( \f 1{q_0} +\f 1 {q_1})] } \big\| 
    \chi_{E_k}  \big\|_{L^{ ( \frac{r s}{q}  )'}}  \bigg]      
\end{align*}
having used \eqref{eqn: 1.3.admis.growth}. 
Taking the logarithm we deduce condition   \eqref{claimedineq8887774ii} for $G$.

As $g^\theta=g$,   by Proposition~\ref{3linesHirschman} we conclude 
  \begin{equation}\label{equ:TSgmLe1}
 \int_{X_0}
  \abs{T_{\theta }(f_{ 1,\theta }^{\ve} ,\dots, 
  f_{ m,\theta }^{\ve} ) }^{\frac{q}{r}}g \; d\mu_0=  G(\theta ) \le  
  \Big(  B_0^{1-\theta} B_1^{\theta}  M(\theta )
  \prod_{\kappa=1}^m
\Big(\big\|f^\kappa\big\|_{L^{p^\kappa } }^{p^\kappa}  +\ve'\Big)^{\f 1 {p^\kappa} } 
  \Big)^{\!\frac{q}{r}}.
  \end{equation}
Inequality \eqref{equ:TSgmLe1} implies that
  \begin{align}
&  \big\|T_{\theta }(f_{\theta }^{1,\ve} ,\dots, f_{\theta }^{m,\ve}  )\big\|_{L^q} \notag \\
  = &\,\,
  \bigg\| \abs{T_{\theta }(f_{ \theta }^{1,\ve} ,\dots, f_{ \theta }^{m,\ve} )}^{\frac{q}{r}}
  \bigg\|_{L^r}^{\frac{r}{q}} \notag \\
  =&\,\,
  \sup  \bigg\{ \!\! \int_{X_0} \abs{T_{\theta}(f_{\theta}^{1,\ve} ,\dots ,  f_{\theta}^{m,\ve}    ) }^{\frac{q}{r}} g  d\mu_0 \! : \, 
   \mbox{$g = \sum_{k=1}^Ka_k\chi_{E_k},$ $a_k>0$, $E_k$ compact}, 
\,   \|g \|_{L^{r'}}= 1
  \bigg\}^{\frac{r}{q}}  \notag    \\
  \le&\, \,  B_0^{1-\theta} B_1^{\theta}  M(\theta ) 
  \prod_{\kappa=1}^m
\Big(\big\|f^\kappa\big\|_{L^{p^\kappa } }^{p^\kappa}  +\ve'\Big)^{\f 1 {p^\kappa} } . \label{aiwueha}
  \end{align}
  
 We also note that a similar argument applied to the log-subharmonic function 
\[
H(z) = \int_{X_0}  \abs{T_{ z}(f_1,\dots , f_m ) }^{\frac{q}{r}}\abs{g^z }\, d\mu_0
\]
 yields the  estimate 
 \[
 |H(\theta) |  =\bigg| \int_{X_0} | T_\theta (f_1,\dots , f_m) |^{\f qr}\, g  \, d\mu_0 
 \bigg| \le \Big(  B_0^{1-\theta} B_1^{\theta} M(\theta) 
 \prod_{\kappa=1}^m \big\| f^\kappa \big\|_{L^{p_0^\kappa}}^{1-\theta} 
 \big\| f^\kappa\big\|_{L^{p_1^\kappa}}^{ \theta}  \Big)^{\! \f qr}, 
 \]
 from which it follows  that 
\begin{equation}\label{kjwakjjkenj}
 \big\| T_{\theta }(f^1,\dots , f^m)\big\|_{L^q} \le B_0^{1-\theta} B_1^{\theta}  M(\theta) 
\prod_{\kappa=1}^m \big\| f^\kappa \big\|_{L^{p_0^\kappa}}^{1-\theta} 
 \big\| f^\kappa\big\|_{L^{p_1^\kappa}}^{ \theta} ,  
 \end{equation}
via a duality argument similar to that leading to \eqref{aiwueha}. 
  
We now make use of   the triangle inequality
  $$
  \| T_{\theta }(f^1,\dots , f^m ) \|_{L^q}^{\min(1,q)}  \le 
  \sum_{\kappa=1}^m     \| T_{\theta }(\dots, f^\kappa - f _{\theta }^{\kappa, \ve},\dots) \|_{L^q}^{\min(1,q)}  +  
  \|T_{\theta }(f _{\theta }^{1,\ve},\dots ,f _{\theta }^{m,\ve} ) \|_{L^q}^{\min(1,q)}  . 
  $$
For the second term on the right above we use \eqref{aiwueha}, while   the first term 
is bounded by a constant multiple of $(\ve^{1-\theta})^{\min(1,q)} $ 
in view of   \eqref{kjwakjjkenj},  and hence it  tends to zero as $\ve\to 
0$.  We deduce \eqref{eqn: 3.1.876543} by letting $\ve\to 0$. 
 
 Finally, if $p_0^\kappa=p_1^\kappa=\nf$ for certain $\kappa$ we factor these $\kappa$'s 
and we consider another multilinear operator of lower degree. For instance  if
$p_0^\kappa=p_1^\kappa=\nf$ exactly when $\kappa \le \la$, we consider the 
operator 
\[
(f^{\la+1}, \dots , f^m) \mapsto T_z(f^1,\dots, f^\la, f ^{\la+1 } , \dots, f ^{m } ) 
\]
which satisfies the initial assumptions with   constants $B_0$ and $B_1$ 
replaced by the original ones multiplied by $\prod_{\kappa=1}^\la \| f^\kappa\|_{\li}$. 
\end{proof}

As we already mentioned in the introduction, an interpolation theorem for analytic families of multilinear operators was proved in \cite{GrafakosMastylo}. The main difference between these results  is that  in \cite{GrafakosMastylo}  the concepts of analyticity and  admissibility condition  are in the pointwise while ours are in the integral sense, as mandated by   applications (see next section). 
Unlike  \eqref{eqn: 1.3.admis.growth} this pointwise  admissibility condition is not easy to check in general, especially when the operators involved do not have  explicit formulae.

\section{A bilinear estimate for Schr\"odinger operators}\label{sec3}

We consider the self-adjoint operator
$$L = -{\rm div}( A \nabla) + V$$
 on $L^2(\RR^n)$ where $A = (a_{kl})$ is a symmetric matrix with real-valued and bounded measurable entries. It is assumed to be elliptic with ellipticity constant $\gamma > 0$, that is
 $$ \sum_{k,l} a_{kl}(x) \xi_k \xi_l \ge \gamma | \xi |^2, \quad a.e.\,  x \in \RR^n, \ \forall \xi= (\xi_1, ..., \xi_n) \in \RR^n.$$
 The potential $V$ is assumed to be nonnegative and locally integrable.\\
By the standard  sesquilinear  form technique, one constructs a self-adjoint realization of $L$. The following theorem was proved by  Dragicevic and  Volberg \cite{DraVolb1}. 
\begin{theorem}\label{th1}
Let ${\bf \Gamma}$  be either $\nabla$ or multiplication by $\sqrt{V}$. Let $p \in (1, \infty) $ and $ p'$ its conjugate number. Then there exists a constant $C_\gamma$, independent of the dimension $n$,  such that 
\begin{equation}\label{eq1}
\int_0^\infty \int_{\RR^n} | {\bf \Gamma} e^{-tL} f (x)| |{\bf \Gamma} e^{-tL}g(x) |\,  dx\, dt\le C_\gamma \max(p,p')  \| f \|_{L^p} \| g \|_{L^{p'}}.
\end{equation}
\end{theorem}
 The constant $C_\gamma$ can be taken to be $C \max(1, \frac{1}{\gamma})$ with $C$ an absolute constant. 

The aim of this section is to prove, under the same assumptions as before, the following result.  
\begin{proposition}\label{pro2}
Let $\alpha, \beta \in [0, \infty)$ and  $1<p<\infty$. Then there exists a constant $ C(\alpha, \beta, \gamma, p)$, independent of $n$,  such that
 \begin{equation}\label{eq2}
\int_0^\infty \int_{\RR^n} | {\bf \Gamma} L^\alpha e^{-tL} f (x)\cdot 
{\bf \Gamma} L^\beta e^{-tL}g(x) |\,  dx\, t^{\alpha + \beta} dt\le C(\alpha, \beta, \gamma, p) 
 \| f \|_{L^p} \| g \|_{L^{p'}}.
\end{equation}
\end{proposition}
This proposition can be viewed  as a weighted version  of the bilinear estimate stated in the previous theorem. More precisely, let
$\omega : (0, \infty)  \to (0, \infty)$ such that $\omega(t)  \sim  t^\eta$ for some $\eta > 0$. Then \eqref{eq2} can be rewritten as 
\begin{equation}\label{eq2-1}
\int_0^\infty \int_{\RR^n} | {\bf \Gamma}  e^{-tL} f (x)\cdot {\bf \Gamma}  e^{-tL}g(x) |\,  dx\, \omega(t)\, dt\le C(\alpha, \eta, \gamma, p)  \| L^{-\alpha}f \|_{L^p} \| L^{-(\eta-\alpha)}g \|_{L^{p'}} 
\end{equation}
for $\alpha \in [0, \eta]$.

\begin{proof}[Proof of Proposition \ref{pro2}]
Define 
$$ 
T_{\alpha, \beta}(f,g)(x,t) = {\bf \Gamma} (tL)^\alpha e^{-tL}f(x)
\cdot  {\bf \Gamma} (tL)^\beta e^{-tL}g(x).
$$
The above proposition can be rephrased as 
$$
 T_{\alpha, \beta} : L^p(\RR^n) \times L^{p'}(\RR^n) \to L^1(\RR^n \times (0, \infty), dx dt)
 $$
is a bounded bilinear operator with norm estimated by $C(\alpha, \beta, \gamma, p)$. 
 
For complex  $z$, we define the bilinear operator
 $$
  T_z (f,g)(x,t) = {\bf \Gamma} (tL)^{\alpha'z} e^{-tL}f(x)
 \cdot  {\bf \Gamma} (tL)^{\beta'z} e^{-tL}g(x),
 $$
 where $\alpha', \beta' \in [0, \infty)$ will be specified later. 
 
 We show that the family $(T_z)$ is analytic in the sense of Definition \ref{DefAFAG}. Let $f, g \in \mathcal C_c(\RR^n)$ and $w \in L^\infty(\RR^n \times (0,\infty))$ a bounded function with compact support $K$. We  prove that
 \begin{equation}\label{eqfunction}
 z \mapsto \int_0^\infty \int_{\RR^n} T_z(f,g)(x,t) w(x,t) \, dx\, dt
 \end{equation}
 is analytic on $\mathbf S$ and continuous on $\overline{\mathbf S}$. 
 
 Note that there exist a compact set $K_0$ of $\RR^n$ and $0 < a < b < \infty$ such that 
 $K \subset K_0 \times [a,b]$. This can be seen by taking $K_0 = p_1(K)$ and $[a,b] = p_2(K)$ where $p_1 : \RR^n \times (0,\infty) \to \RR^n$ and 
 $p_2 : \RR^n \times (0,\infty) \to (0, \infty)$ are the first and second projections. These functions are continuous and hence $p_1(K)$ and $p_2(K)$ are compact sets of $\RR^n$ and $(0, \infty)$, respectively. By arguing by contradiction, it is easy to see that $a > 0$. In particular,  the function in 
 \eqref{eqfunction}  coincides with
 $$ 
 z \mapsto \int_a^b \left\langle {\bf \Gamma} (tL)^{\alpha' z} e^{-tL} f, w(\cdot,t) {\bf \Gamma} (tL)^{\beta' z} e^{-tL} g \right\rangle_{L^2}\, dt.
 $$
 
 Note that by ellipticity and the fact that $V$ is nonnegative,
 $$
  \| \nabla u \|_{L^2}^2 \le \frac{1}{\gamma} \int_{\RR^n} L u \,\, \overline{u} \, dx =  \frac{1}{\gamma} \big\| L^{1/2} u \big\|_{L^2}^2 \quad {\rm and} \quad 
 \big\| \sqrt{V} u \big\|_{L^2}^2 \le \big\| L^{1/2} u \big\|_{L^2}^2.
 $$
 Hence
 \begin{equation}\label{eq1m}
 \| {\bf \Gamma} u \|_{L^2}^2 \le \max\Big(1, \frac{1}{\gamma}\Big) \big\| L^{1/2} u \big\|_{L^2}^2.
 \end{equation}
 
 Recall that for every $h \in D(L)$, the function $z \mapsto L^z h$ is analytic on $\mathbf S$ and continuous on $\overline {\mathbf S}$ (see e.g. \cite{Haase}, Proposition 3.1.1, b)). Since the operator  $\Gamma e^{-tL}$ is bounded on $L^2(\RR^n)$  for every $t > 0$ (see \eqref{eq1m}), it follows that  the function
$$ z \mapsto \left\langle {\bf \Gamma} (tL)^{\alpha' z} e^{-tL} f, w(\cdot,t) {\bf \Gamma} (tL)^{\beta' z} e^{-tL} g \right\rangle_{L^2}$$
is analytic on $\mathbf S$ and continuous on $\overline{\mathbf S}$. It remains to bound in a neighborhood of each $z_0  \in \overline{\mathbf S}$
 this function by some function $\psi(t)$ which is integrable on $[a,b]$ and then obtain the desired conclusion for the function in \eqref{eqfunction}.  

 By the Cauchy-Schwarz  inequality and  \eqref{eq1m} we write 
\begin{eqnarray}
&&\hspace{-1cm}  \left| \left\langle {\bf \Gamma} (tL)^{\alpha' z} e^{-tL} f, w(\cdot,t) {\bf \Gamma} (tL)^{\beta' z} e^{-tL} g \right\rangle_{L^2} \right| \label{eqabc}\\
&\le& \| w \|_{L^\infty} \left\| {\bf \Gamma} (tL)^{\alpha' {\real z}} e^{-tL}  (tL)^{i \alpha'{\imag z}}f  \right\|_{L^2} 
\left\| {\bf \Gamma} (tL)^{\beta' {\real z}} e^{-tL}  (tL)^{i \beta'{\imag z}}g \right\|_{L^2} \nonumber\\
&\le& \| w \|_{L^\infty} \max\Big(1, \frac{1}{\gamma}\Big)\left\| L^{1/2} (tL)^{\alpha' {\real z}} e^{-tL}  (tL)^{i \alpha'{\imag z}}f  \right\|_{L^2} 
\left\| L^{1/2} (tL)^{\beta' {\real z}} e^{-tL}  (tL)^{i \beta'{\imag z}}g \right\|_{L^2} .\nonumber
\end{eqnarray}
The standard functional calculus for self-adjoint operators, i.e., 
$$ 
\| \phi(L) h \|_{L^2} \le \sup_{\lambda > 0} | \phi(\lambda) | \| h \|_{L^2},
$$
 gives
\begin{equation}\label{eqfc}
 \| (tL)^{\alpha' \real z} e^{-tL} h \|_{L^2} \le e^{ { \alpha' \real z}(\log({ \alpha' \real z}) - 1)} \|h\|_{L^2}.
 \end{equation}
Clearly the term on the right hand side of \eqref{eqfc} is uniformly bounded in $z$ in a bounded neighborhood $W_0$ of a fixed $z_0  \in \overline{\mathbf S}$.
It follows from this and the estimates in \eqref{eqabc} that 
$$\left| \left\langle {\bf \Gamma} (tL)^{\alpha' z} e^{-tL} f, w(\cdot,t) {\bf \Gamma} (tL)^{\beta' z} e^{-tL} g \right\rangle_{L^2} \right|
\le
\frac{C \|w\|_{L^\infty}} {t} \| f \|_{L^2} \| g \|_{L^2}, \quad z \in W_0.
$$
This function is integrable on $[a,b]$  and hence by the dominated convergence theorem we obtain that the function in \eqref{eqfunction} is analytic at $z_0 \in {\mathbf S}$ and continuous at $z_0 \in \overline{\mathbf S}$.

 Next, we prove the admissibility condition. For $f, g \in L^2(\RR^n)$ and $z = r+ is \in \overline{\mathcal S}$,
 \begin{eqnarray*}
 &&\hspace{-1cm} \left\| T_z(f,g) \right\|_{L^1(\RR^n \times (0, \infty))} \\
 &=& \int_0^\infty \int_{\RR^n} | {\bf \Gamma} (tL)^{\alpha' z} e^{-tL}f(x)
 \cdot  {\bf \Gamma} (tL)^{\beta'z} e^{-tL}g(x) |\, dx\,dt\\
 &= & \int_0^\infty \int_{\RR^n} | {\bf \Gamma} (tL)^{\alpha'r } e^{-tL} L^{is \alpha'} f 
 \cdot  {\bf \Gamma} (tL)^{r\beta'} e^{-tL} L^{is\beta'} g |\, dx\,dt\\
 &\le& \left\| \left( \int_0^\infty |  {\bf \Gamma} (tL)^{\alpha'r} e^{-tL} L^{is\alpha'} f |^2\, dt  \right)^{1/2} \right\|_{L^2}  \left\| \left( \int_0^\infty |  {\bf \Gamma} (tL)^{\beta'r} e^{-tL} L^{is\beta'} f |^2\, dt  \right)^{1/2} \right\|_{L^2}.
 \end{eqnarray*}
 We estimate the latest terms using the standard functional calculus for the self-adjoint operator $L$ on $L^2(\RR^n)$.  Using \eqref{eq1m} we have 
{\allowdisplaybreaks \begin{eqnarray*}
 &&\hspace{-1cm} \left\| \left( \int_0^\infty |  {\bf \Gamma} (tL)^{\alpha'r} e^{-tL} L^{is\alpha'} f |^2\, dt  \right)^{1/2} \right\|_{L^2}^2  \\
 &=& \int_0^\infty \left\| {\bf \Gamma} (tL)^{\alpha'r} e^{-tL} L^{is\alpha'} f  \right\|_{L^2}^2 \, dt\\
 &\le& \max\Big(1, \frac{1}{\gamma}\Big) \int_0^\infty \left\| L^{1/2} (tL)^{\alpha'r} e^{-tL} L^{is\alpha'} f  \right\|_{L^2}^2 \, dt\\
 &=& \max\Big(1, \frac{1}{\gamma}\Big) \int_0^\infty \left\langle L^{1/2} (tL)^{\alpha'r} e^{-tL} L^{is\alpha'} f , L^{1/2} (tL)^{\alpha'r} e^{-tL} L^{is\alpha'} f \right\rangle_{L^2} \, dt\\
 &=& \max\Big(1, \frac{1}{\gamma}\Big) \int_0^\infty \left\langle  (tL)^{2\alpha'r +1} e^{-2tL} L^{is\alpha'} f,  L^{is\alpha'} f \right\rangle_{L^2}\, \frac{dt}{t}\\
 &=&\max\Big(1, \frac{1}{\gamma}\Big) \left\langle \int_0^\infty  (tL)^{2\alpha'r +1} e^{-2tL} L^{is\alpha'} f \, \frac{dt}{t},  L^{is\alpha'} f \right\rangle_{L^2}\\
 &\le& \max\Big(1, \frac{1}{\gamma}\Big)\left\| \int_0^\infty   (tL)^{2\alpha'r +1} e^{-2tL} L^{is\alpha'} f \, \frac{dt}{t} \right\|_{L^2} \left\| L^{is\alpha'} f \right\|_{L^2}.
 \end{eqnarray*}
 }
 Using again the functional calculus, we have  $\| L^{is\alpha'} f \|_2 = \| f \|_2$ and 
 \begin{eqnarray*}
  \left\| \int_0^\infty  (tL)^{2\alpha'r +1} e^{-2tL} L^{is\alpha'} f \, \frac{dt}{t} \right\|_{L^2} 
 &\le& \sup_{\lambda > 0} \left| \int_0^\infty  (t\lambda)^{2\alpha'r +1} e^{-2t\lambda} \,  \frac{dt}{t} \right|\,   \| L^{is\alpha'} f \|_{L^2} \\
 &=&  \int_0^\infty  t^{2\alpha'r} e^{-2t} \, dt\,    \|  f \|_{L^2}\\
 &=&  2^{-2\alpha' r -1}  \Gamma(2\alpha' r +1)\, \| f \|_{L^2}.
 \end{eqnarray*}
 Thus we obtain for all $z = r+ is \in \overline{\mathcal S}$
 \begin{equation}\label{eq3}
 \| T_z(f,g) \|_{L^1(\RR^n \times (0, \infty))} \le  
 \f{ \max(1, \frac{1}{\gamma}) }{ 2^{ (\alpha' + \beta') r +1} } \sqrt{ \Gamma(2 \alpha' r +1) \Gamma(2\beta' r + 1)}\,  \| f \|_{L^2} \|g \|_{L^2}.
 \end{equation}
 In particular,
  \begin{equation}\label{eq4}
 \| T_z(f,g) \|_{L^1(\RR^n \times (0, \infty))} \le \max\Big(1, \frac{1}{\gamma}\Big)
  \sqrt{ \Gamma(2 \alpha'  +1) \Gamma(2\beta'  + 1)}\,  \| f \|_{L^2} \|g \|_{L^2}
 \end{equation}
   for all $f, g \in L^2(\RR^n)$ and all $z \in \overline{\mathcal S}$. This proves that the analytic family of bilinear operators $T_z$ is of admissible  growth in the sense of Definition \ref{DefAFAG}. 
   
   The particular case of \eqref{eq3} for  $z = 1+ is$ yields
   \begin{equation}\label{eq5}
 \| T_{1+is}(f,g) \|_{L^1(\RR^n \times (0, \infty))} \le \f{ \max(1, \frac{1}{\gamma}) }{ 
  2^{(\alpha' + \beta')  +1}}  \sqrt{ \Gamma(2 \alpha'  +1) \Gamma(2\beta'  + 1)}\,  \| f \|_{L^2} \|g \|_{L^2}.
 \end{equation}

Next, we estimate the $L^1$-norm of $T_{is}(f,g)$. Let $p_1 \in (1, 2)$ be a fixed number and let $f \in L^{p_1}(\RR^n)$ and $g \in L^{p_1'}(\RR^n)$. By Theorem \ref{th1} we obtain 
\begin{eqnarray*}
\| T_{is}(f,g) \|_{L^1(\RR^n \times (0, \infty))} 
 &=& \int_0^\infty \int_{\RR^n} | {\bf \Gamma} e^{-tL} L^{is \alpha'} f(x)\cdot 
  {\bf \Gamma}  e^{-tL} L^{is \beta'}g(x) |\, dx\,dt\\
&\le& C_\gamma p_1' \| L^{is\alpha'} f \|_{L^{p_1}} \| L^{is\beta'} g \|_{L^{p_1'}}.
\end{eqnarray*}
Since the semigroup $(e^{-tL})$ is sub-Markovian and symmetric, $L$ has a holomorphic functional calculus on $L^q(\RR^n)$ for all 
$q \in (1, \infty)$ (cf. \cite{Cowling}, or \cite{Carbonaro}). For imaginary  powers, it follows from  these last two references that there exists a constant $C(p_1)$, independent of $n$, such that for all $s \in \RR$,
\begin{equation}\label{eq6} 
\| L^{is\alpha'} f \|_{L^{p_1} } \le C(p_1) e^{\frac{\pi}{2} |s| \alpha'} \| f \|_{L^{p_1}}.
\end{equation}
Therefore,
\begin{equation}\label{eq7}
\| T_{is}(f,g) \|_{L^1(\RR^n \times (0, \infty))} \le C(\gamma, p_1) e^{\frac{\pi}{2}(\alpha' + \beta') |s|} \| f \|_{L^{p_1}} \| g \|_{L^{p_1'}}.
\end{equation}

We are now in the position  to apply Theorem \ref{thm-multilinear}. It follows from \eqref{eq5} and \eqref{eq7} that for $\theta \in (0, 1)$ and $\frac{1}{p_\theta} = \frac{\theta}{2} + \frac{1-\theta}{p_1}$ we have 
 \begin{equation}\label{eq8}
 \| T_{\theta}(f,g) \|_{L^1(\RR^n \times (0, \infty))} \le c_\theta M(\theta) 
 \| f \|_{L^{ p_\theta} } \| g \|_{L^{ p_\theta'}} 
 \end{equation}
 with
 $$
 c_\theta = \left( \max\Big(1, \frac{1}{\gamma}\Big) 2^{-\alpha' -\beta'-1} \sqrt{ \Gamma(2 \alpha'  +1) \Gamma(2\beta'  + 1)} \right)^\theta C(\gamma,p_1)^{1-\theta}
 $$
 and 
 $$ 
 M(\theta) = \exp\left\{ \frac{\sin(\pi \theta)}{2} \frac{\pi}{2}(\alpha' + \beta') \int_{-\infty}^{+\infty} \frac{  |s|}{\cosh(\pi s) + \cos(\pi \theta)} \, ds \right\}.
 $$
Finally, for any $p \in (1, 2)$ we choose $p_1 < p$,  $\alpha' = \frac{\alpha}{\theta}$, $\beta' = \frac{\beta}{\theta}$ and set $\theta = \frac{p-p_1}{2-p_1}\frac{2}{p}$ so that $p_\theta = p$ and $T_\theta = T_{\alpha, \beta}$. The proposition follows from \eqref{eq8}.  
 \end{proof}
 
 \begin{remark}
It is an interesting question to understand  for which functions $F$ and $G$ one has 
 $$
  \int_0^\infty \int_{\RR^n} | {\bf \Gamma} F(tL)  f (x)\cdot {\bf \Gamma} G(tL) g(x) |\,  dx\, dt\le C(F,G, p)  \| f \|_{L^p} \| g \|_{L^{p'}}.
  $$
Note that the term on the left hand side is bounded by the product  
$$
\left\| \left(\int_0^\infty | {\bf \Gamma} F(tL) f |^2\, dt \right)^{1/2} \right\|_{L^p} \left\| \left(\int_0^\infty | {\bf \Gamma} G(tL) g |^2\, dt \right)^{1/2} \right\|_{L^{p'}}.
$$
The Littlewood-Paley-Stein functional $\left(\int_0^\infty | {\bf \Gamma} F(tL) f |^2\, dt \right)^{1/2}$
 is bounded on $L^p(\RR^n)$ for $p \in (1,2]$ as soon as $F$ is holomorphic in a certain  sector (with angle depending on $p$) and decays faster that $\frac{1}{\sqrt{|z|}}$ at $\infty$ (see \cite{CometxOuhabaz}). Thus the first term in the above product is fine for $p \in (1, 2]$. However the second term could  be unbounded on $L^{p'}(\RR^n)$ even if $L = \Delta + V$ ($V \not= 0$)  and $G(z) = e^{-z}$. See again \cite{CometxOuhabaz}. 
\end{remark}

\section{Appendix: Log-subharmonic   functions on the plane}\label{sec4}
A locally integrable function $f$ on an open subset $ O$ of 
the complex plane with values in $[-\nf,\nf)$ is called 
subharmonic if it is upper semicontinuous, i.e., 
$
\limsup_{w\to z} f(w) \le f(z)
$
for every $z\in O$ and satisfies 
\bee\lab{MVPSEMINGRAN}
f(z) \le \f{1}{|B(z,r)|} \int_{B(z,r)} f(w)\, dw 
\eee
for any $z\in O$ and every $r>0$ such that $B(z,r)\subset O$. If $f\in \mathcal C^2$, then 
the above condition is equivalent to $\De f\ge 0$.  A function is called log-subharmonic 
if it is nonnegative and its logarithm is subharmonic.

\begin{lemma}\lab{logsubh}
The sum of two log-subharmonic functions is log-subharmonic. 
\end{lemma}

\begin{proof}
Let $\vp(x,y)= \log (e^x+e^y)$ defined on $\RR^2$. Then $\vp$ is obviously 
increasing in each variable and is a convex function of both  variables.

Suppose that $F,G$ are subharmonic functions on an open subset of the
complex  plane. Then the 
fact that $\vp$ is increasing in each variable and Jensen's inequality (which can be used
since $\vp$ is convex) gives
\begin{eqnarray*}
\vp(F(z), G(z)) & \le&  \vp \bigg( \f{1}{|B(z,r)|} \int_{B(z,r)}  F(w)\, dw \,\,\, , \,\,
\f{1}{|B(z,r)|} \int_{B(z,r)}  G(w)\, dw  \bigg)  \\
& \le&  \f{1}{|B(z,r)|} \int_{B(z,r)} \vp(  F(w),G(w))\, dw\, ,
\end{eqnarray*}
which implies that $\vp(F(z), G(z))$ is subharmonic. 
Now let $f,g$ be log-subharmonic functions. Writing $f=e^F$ and $g=e^G$, 
then $\log (f+g) = \vp(F, G)$. But $\vp(F, G)$   was shown to be subharmonic, thus 
  $\log (f+g)$ is also subharmonic. 
 \end{proof}

We review a couple of facts from the theory of analytic functions with values in Banach spaces. 
Let $\mathcal B$ be a Banach space and let $\mathbf f$ be a mapping from an open 
subset $U $ of $\CC$ to $\mathcal B$. We say that $\mathbf f$ is analytic 
if
$$
\mathbf f'(z_0) =\lim_{z\to z_0} \f{\mathbf f(z)- \mathbf f(z_0)}{z-z_0}
$$
exists in the norm of $\mathcal B$. 

\begin{theorem}\label{BanachAnalytic}
Let $\mathbf f$ be a mapping from an  open 
subset $U $ of $\CC$ to a Banach space $\mathcal B$. Then $\mathbf f$  is 
analytic if and only if   
 for every bounded linear functional $\Lambda$ in $\mathcal B$ we have 
$$
\lim_{z\to z_0}\Lambda\bigg(   \f{\mathbf f(z)- \mathbf f(z_0)}{z-z_0}\bigg) 
$$
exists in $\CC$.  
\end{theorem}

Log-subharmonic can be generated from $L^1(X)$-valued analytic functions
in terms of the subsequent   lemma. 

\begin{lemma}\lab{SWTohoku} \cite[Lemma 2]{SteinWeissTohoku}
Let $(X,\mu)$ be a measure space 
with $\mu(X)<\nf$ and let $V $ be a complex-valued function 
defined on $X\times S$  such that   the  mapping  
 $z\mapsto V(\cdot ,z)$ from $S$ to $L^1(X)$ is 
a Banach-valued analytic function.  Then the function
\[
z\mapsto F(z)=\int_X |V(x,z)|^q\, d\mu(x)
\]
is log-subharmonic for any $0<q\le 1$. 
\end{lemma}

\end{document}